\documentclass{gtart} 
\usepackage{amsmath, amssymb, latexsym, amscd, graphicx}

\usepackage[inner=25mm, outer=25mm, textheight=225mm]{geometry}

\def\bkC{{\rm \kern.24em \vrule width.05em height1.4ex depth-.05ex \kern-.26em C}}
\def\C{\bkC}
\def\bksC{{\rm \kern.24em \vrule width.05em height1ex depth-.05ex \kern-.26em C}}

\def\bkH{{\rm I\kern-.22em H}}
\def\H{\bkH}                    
\def\bkR{{\rm I\kern-.17em R}}
\def\RR{\bkR}
\def\bkZ{{\rm Z\kern-.32em Z}}
\def\Z{\bkZ}
\def\bksZ{{\rm Z\kern-.22em Z}}

\def\SL{SL_2(\C)}
\def\PSL{PSL_2(\C)}
\def\Ei{\mathfrak{E}}

\def\X{\widetilde{\mathfrak{X}}}
\def\PX{\mathfrak{X}}

\def\longitude{\mathcal{L}}
\def\m{\mathcal{M}}

\DeclareMathOperator{\qe}{q}

\DeclareMathOperator{\de}{d}
\DeclareMathOperator{\Vol}{Vol}
\DeclareMathOperator{\vol}{vol}

\newtheorem{thm}{Theorem}
\newtheorem{lem}[thm]{Lemma}
\newtheorem{cor}[thm]{Corollary}
\newtheorem{pro}[thm]{Proposition}

\begin{document}
\title{A birationality result for character varieties} 
\authors{Ben Klaff and Stephan Tillmann}

\begin{abstract}
Let $M$ be an orientable, cusped hyperbolic 3--manifold of finite volume. We show that the restriction map $r\co \PX_0 \to \PX(\partial M)$ from a Dehn surgery component in the $\PSL$--character variety of $M$ to the character variety of the boundary of $M$ is a birational isomorphism onto its image. This generalises a result by Nathan Dunfield. A key step in our proof is the exactness of Craig Hodgson's volume differential on the eigenvalue variety.
\end{abstract}

\maketitle


\section{Introduction}

Let $M$ be an orientable cusped complete hyperbolic 3--manifold of finite volume. There is a discrete and faithful representation of $\pi_1(M)$ into $\PSL$, and its character $\chi_0$ is known to be a smooth point of the $\PSL$--character variety $\PX(M).$ The irreducible component $\PX_0$ of $\PX(M)$ containing $\chi_0$ is called a \emph{Dehn surgery component} of $\PX(M)$. (There is also a Dehn surgery component, possibly the same, containing the complex conjugate character $\overline{\chi}_0.$)
The inclusion map $\partial M \to M$ induces a restriction map $$r : \PX(M) \to\PX(\partial M).$$ It is shown in \cite{tillus_ei} that the restriction of $r$ to $\PX_0$ is finite--to--one onto its image. In the case where $M$ has only one cusp, Dunfield has shown in \cite{du2} that $r\co \PX_0 \to \PX(\partial M)$ has degree one onto its image using Thurston's Hyperbolic Dehn Surgery Theorem and a Volume Rigidity Theorem attributed to Gromov, Thurston and Goldman. This note generalises Dunfield's result to manifolds with an arbitrary number of cusps: 

\begin{thm}\label{thm:birational}
Let $M$ be an orientable, non--compact, complete hyperbolic 3--manifold of finite volume. Let $\PX_0$ be a Dehn surgery component in the $\PSL$--character variety of $M$. Then the restriction map $r : \PX_0 \to \PX(\partial M)$ is a birational isomorphism onto its image.
\end{thm}

\begin{cor}\label{cor:birational}
Let $M$ be an orientable, non--compact, complete hyperbolic 3--manifold of finite volume with $h$ cusps. Let $\X_0$ be a Dehn surgery component in the $\SL$--character variety of $M$. Then the restriction map $r : \X_0 \to \X(\partial M)$ has degree at most $2^{-h}|H^1(M, \Z_2)|$ onto its image. In particular, if $H^1(M, \Z_2) \cong \Z_2^h,$ then the map is a birational isomorphism.
\end{cor}

The proof of Theorem \ref{thm:birational} is a generalisation of Dunfield's argument. Our new contributions are the construction of an explicit Zariski dense set in $\X_0$ using a Dehn surgery argument, and a proof of the fact that Hodgson's \emph{volume differential} \cite{ho} is an exact form on the \emph{eigenvalue variety} \cite{tillus_ei}. The eigenvalue variety $\Ei(M)$ of a multi-cusped hyperbolic 3-manifold is a natural generalisation of the curve defined by the $A$--polynomial \cite{ccgls} for a one-cusped hyperbolic 3--manifold. Choose a basis $(\m_i, \longitude_i)$ for $H_1(T_i),$ where $T_i$ is a torus cross-section of the i--th cusp, and suppose that $M$ has $h$ cusps. Then the eigenvalue variety is the Zariski closure of the set of points $(m_1, l_1, \ldots, m_h, l_h)$ in $(\C\setminus \{0\})^{2h}$ with the property that there is a representation $\rho\co \pi_1(M)\to \SL$ such that ${\rho}(\m_i)$ and ${\rho}(\longitude_i)$ have eigenvalues $m_i$ and $l_i$ respectively with respect to a common eigenvector. In this notation, Hodgson's volume differential is the 1--form
$$\eta = -\sum_{i=1}^h \Big(\log |l_{i}| \de \arg m_i - \log |m_{i}| \de \arg l_i\Big)$$
on the eigenvalue variety, 
where $M$ is oriented, each boundary component is given the induced orientation and $(\m_i, \longitude_i)$ is a left-handed basis with respect to this orientation.

The geometric significance of this form is as follows (see \S\S2--3 for details and references). A representation $\rho\co \pi_1(M)\to\SL$ determines a pseudo-hyperbolic structure on $M,$ i.e.\thinspace a pseudo-developing map from its universal cover $\widetilde{M}$ to hyperbolic space $\H^3$ with holonomy $\rho.$ Taking the volume of this structure (which is possibly zero or negative) gives a function on the $\SL$--character variety, $\Vol_M\co \X(M)\to \RR.$ Dunfield establishes fundamental results about this function, which carry over to our setting. These are recalled in \S\ref{sec:Volume results}. Work of Hodgson \cite{ho} and Neumann and Zagier \cite{nz} shows that $d\Vol_M=\omega$ (see \S4.5 of \cite{ccgls} and \S\ref{sec:Hodgson's volume differential} below), where $\omega$ is the 1--form $\eta$ interpreted as a form on the complement of a suitable subvariety of $\X_0.$ We prove in \S\ref{sec:Exactness of volume form} that $\eta$ is also exact on the complement of a suitable subvariety of the eigenvalue variety. This is then used to show that the map $r$ on $\PSL$--characters has degree one by studying it on a suitable Zariski dense set of characters. This set is defined in \S\ref{sec:A Zariski dense set}. The proofs of Theorem~\ref{thm:birational} and Corollary~\ref{cor:birational} are put together in \S\ref{sec:Proofs}.

Dunfield~\cite{du2} applies his birationality theorem to settle a conjecture due to      
Boyer and Zhang concerning cyclic surgeries of certain hyperbolic knots.        
Theorem 1 can be used to understand Dehn surgery spaces of multi-cusped         
hyperbolic 3-manifolds; see Klaff \cite{k} for an application.

\rk{Acknowledgements} The authors thank Steven Boyer, Daryl Cooper and Craig Hodgson for
helpful conversations, and the referee for very useful comments, corrections and suggestions. The second author is partially supported by ARC grant DP130103694. 


\section{Volume of representations}
\label{sec:Volume results}

We refer the reader to \cite{cs, bozh} for standard facts about character varieties, and to \cite{mu} as a reference for algebraic geometry.
This section summarises the material we need from Dunfield~\cite{du2}.

Let $M$ be a complete hyperbolic 3--manifold of finite volume with universal cover $\widetilde{M}.$ If $M$ is closed, then the volume of a representation $\rho\co\pi_1(M)\to\PSL$ is well-defined by letting $\Vol_M(\rho)=\int_F f_\rho^*(\vol_{\H^3})$, where $f_\rho\co\widetilde{M}\to\H^3$ is any smooth equivariant map, $\vol_{\H^3}$ is the usual volume form on hyperbolic space and $F$ is any fundamental domain for $M$. The volume is an invariant of the conjugacy class of a representation, and Dunfield proves the following:
\begin{thm}[Gromov-Thurston-Goldman, in Dunfield \cite{du2}, Theorem 6.1]\label{volume-complete}
Let $M$ be a closed, hyperbolic 3--manifold of finite volume, and $\chi \in \PX(M)$. If $\rho$ is a representation with character $\chi$ and $\Vol_{M}(\rho) = \pm \Vol(M)$, then $\rho$ is discrete and faithful.
\end{thm}

If $M$ is not closed, Dunfield defines the volume of a representation $\rho:\pi_1(M)\to\PSL$ with respect to a so-called \emph{pseudo-developing map} $f_\rho:\widehat{M}\to\overline{\H}^3.$  A pseudo-developing map is a $\rho$--equivariant map which satisfies two technical conditions which ensure that the integral defining the volume of $\rho$ is both finite and independent of the chosen pseudo-developing map. We recall the definition.

The space $\overline{\H}^3$ is the usual compactification of $\H^3$ obtained by adding the sphere at infinity. The space $\widehat{M}$ is obtained from $\widetilde{M}$ by adding countably many points as follows. The manifold $M$ is naturally the interior of a compact manifold $N$ with boundary. Choose a collar neighbourhood $T^2 \times [0, \infty]$ for each boundary component of $N,$ where we assume that $T^2\times\{\infty\} \subseteq \partial N.$ Then $\overline{M}$ is the quotient space obtained by collapsing each $T^2\times \{\infty\}$ to a point. There is a natural inclusion $M \hookrightarrow \overline{M}$ and  $\overline{M}\setminus M$ is a finite collection of points, one for each cusp of $M.$ The construction for the universal cover is analogous. Lift the product structure at each boundary component of $N$ to $\widetilde{N}.$ Each connected component of one of the collar neighbourhoods in $N$ is of the form $\RR^2 \times [0,\infty],$ and $\widehat{M}$ is obtained by collapsing each $\RR^2\times \{\infty\}$ to a point.
Note that there is a natural map $\widehat{M} \to \overline{M},$ and each $v\in \widehat{M}\setminus \widetilde{M},$ has a neighbourhood of the form $N_v = (P_v \times [0,\infty)) \cup \{v\}.$ The action of $\pi_1(M)$ by deck transformations extends naturally to $\widehat{M}$ and preserves the chosen product structure of the cusps. With this notation, a $\rho$--equivariant map $f_\rho:\widehat{M}\to\overline{\H}^3$ is a pseudo-developing map if it satisfies the following two conditions:
\begin{enumerate}
\item $f_\rho(\widetilde{M}) \subseteq \H^3$ and $f_\rho(\widehat{M}\setminus \widetilde{M}) \subseteq \partial \H^3;$ and
\item for each $v\in \widehat{M}\setminus \widetilde{M},$ $f_\rho$ maps each ray $\{p\} \times [0, \infty)$ in $N_v$ to a geodesic ray in $\H^3$ with ideal endpoint $f_\rho(v)$ and parameterises this ray by arc-length.
\end{enumerate}

Given a pseudo-developing map $f_\rho:\widehat{M}\to\overline{\H}^3,$ define $\Vol_M(\rho, f_\rho)=\int_F f_\rho^*(\vol_{\H^3}),$ where, as above, $F$ is a chosen fundamental domain. In fact, Dunfield takes the absolute value of this integral, but for our purposes, it will be more convenient to work with a \emph{signed} volume that takes orientation into account. Also, Dunfield only considers hyperbolic 3--manifolds with one cusp. However, a careful examination of the material in \S2.5 of \cite{du2} reveals that it applies to multi--cusped hyperbolic 3--manifolds. The following results from \cite{du2} are therefore at our disposal:

\begin{lem}[Dunfield \cite{du2}, Lemma 2.5.2]
The function $\Vol_M\co\PX_0\to\RR$ defined by taking $\Vol_M(\chi)=\Vol_M(\rho, f_\rho)$, where $\rho$ is any representation with character $\chi$ and $f_\rho$ is any pseudo-developing map for $\rho$, is well--defined.
\end{lem}

\begin{lem}[Dunfield \cite{du2}, Lemma 2.5.4]\label{volume-factors}
  Let $M$ be a complete cusped hyperbolic 3--manifold of finite volume, and
  let $N$ be a closed hyperbolic 3--manifold obtained by Dehn filling on $M$.
  Assume that $\rho$ is a representation of $\pi_1(M)$ which factors through a
  representation $\rho'$ of $\pi_1(N)$. Then $\Vol_{N}(\rho') = \Vol_{M}(\rho, f_\rho)$, where $f_\rho$ is any pseudo-developing map for $\rho$.
\end{lem}

For a representation $\rho\co\pi_1(M)\to \SL,$ we define the volume of $\rho$ to be the volume of its composition with the quotient map $\SL\to\PSL.$


\section{Hodgson's volume differential}
\label{sec:Hodgson's volume differential}

For the material of this section, which is well-known to experts, we refer the reader to Hodgson~\cite{ho}, Neumann--Zagier~\cite{nz} and Cooper-Culler-Gillet-Long-Shalen \cite{ccgls}. Complete details can be found in \cite{hrst2014}, and we will only give an overview.

Consider lifts $\tilde{\chi}_0$ of $\chi_0$ and $\X_0$ of $\PX_0$ to the $\SL$--character variety. Also denote $\tilde{\rho}_0$ a lift of $\tilde{\chi}_0$ to the $\SL$--representation variety. We may choose a fundamental domain $D_0$ for the action of $\tilde{\rho}_0$ consisting of a union of convex ideal hyperbolic polyhedra (see Epstein and Penner~\cite{ep}). The ideal vertices of the polyhedra correspond to fixed points of peripheral subgroups. Given a smooth 1--parameter family of representations $\tilde{\rho}_t$ sufficiently close to $\tilde{\rho}_0,$ we obtain a smooth 1--parameter family of fundamental domains $D_t$ obtained by small deformations of $D_0,$ because all peripheral subgroups have images not contained in $\{\pm E\}$ near $\tilde{\rho}_0$ and so the fixed point sets vary smoothly. Writing $\tilde{\chi}_t$ for the character of $\tilde{\rho}_t,$ we have $\Vol_M(\tilde{\chi}_t) = \Vol_M(\tilde{\rho}_t) =\Vol(D_t),$ noting that there is a natural definition of a pseudo-developing map at the complete structure, and that this can be deformed for nearby representations using $D_t.$

Choose a basis $\{ \m_i, \longitude_i\}$ for $H_1(T_i),$ where $T_i$ is a torus cross-section of the i--th cusp, and we use the same orientation conventions as in the introduction. The eigenvalues $m_i(t)$ and $l_i(t)$ associated to a common eigenvector of $\tilde{\rho}_t(\m_i)$ and $\tilde{\rho}_t(\longitude_i)$ vary smoothly with $t.$ Hodgson~\cite{ho} computes the derivative of volume using the Milnor-Schl\"afli formula~\cite{Mil} for the derivative of volume of a smooth 1--parameter family of hyperbolic polyhedra with ideal vertices, obtaining:
$$\frac{d}{dt}\Vol_M(\tilde{\chi}_t) = -\sum_{i=1}^h \Big(\log |l_{i}(t)| \frac{d}{dt}\arg m_i(t) - \log |m_{i}(t)| \frac{d}{dt}\arg l_i(t)\Big).$$

As discussed in \cite{ho, ccgls}, the above application of the Schl\"afli formula can be modified to show that the above formula holds at each point of the character variety (not just on $\X_0$), except possibly at the points where a peripheral subgroup is contained in $\{\pm E\}.$ The key idea is to allow both ideal and finite vertices, and negatively oriented or flat polyhedra (whose volume is negative or zero respectively). The main issue at $\{\pm E\}$ is that the fixed point set may not converge, and so the above argument of deforming polyhedra with ideal vertices may not apply. To avoid this situation, let 
\begin{equation}\label{def:V}
V = \X_0 \cap \bigcup_{i=1,\ldots,h} \{\chi(\m_i)^2 = \chi(\longitude_i)^2=4\}.
\end{equation}
Then $V$ is a proper subvariety of $\X_0,$ and the function $\Vol_M\co \X_0 \setminus V \to \RR$ is smooth. 


\section{Exactness of the volume differential}
\label{sec:Exactness of volume form}

Given $\gamma\in\pi_1(M)$, there is a rational function $I_{\gamma}:\X_0\to\C$ defined by $\tilde{\chi}\to\tilde{\chi}(\gamma)$. A well-known consequence of Thurston's Hyperbolic Dehn Surgery Theorem and its proof is that $\tilde{\chi}_0$ is a smooth point of $\tilde{X}(M),$ and that the map $I \co{\X}_0 \to \C^h$ defined by $I(\tilde{\chi}) = (I_{\m_1},...,I_{\m_h})$ maps a small open neighbourhood of $\tilde{\chi}_0$
analytically to an open neighbourhood of $(\epsilon_1 2,...,\epsilon_h 2)$, where $I_{\m_{i}}(\tilde{\chi}_{0}) = \epsilon_i 2$ with $\epsilon_i\in \{\pm 1\}.$ See, for instance Thurston \cite{t} (\S5.8), Neumann--Zagier \cite{nz} and Porti \cite{Por97} (Corollaries 3.27 and 3.28). 

This implies that there is a simply connected neighbourhood $V_0 \subset \X_0$ of $\tilde{\chi}_0$ with the property that the volume of any character $\tilde{\chi} \in V_0$ is independent of the path of integration chosen from $\tilde{\chi}_0$ to $\tilde{\chi}$ inside $V_0$ and, moreover, that the volume of $\tilde{\chi}$ only depends on its image under the restriction map $r\co  \X_0 \to \X(\partial M),$ because the latter can be imbued with the coordinates $(I_{\m_1}, I_{\longitude_1}, I_{\m_1\longitude_1}, \ldots, I_{\m_h}, I_{\longitude_h}, I_{\m_h\longitude_h}).$ Note that this observation is not valid globally; for instance, at the character of the discrete and faithful representation, we have $r(\chi_0) = r (\overline{\chi}_0)$ but $\Vol_{M} (\overline{\chi}_0) = -\Vol_{M}(\chi_0) = - \Vol(M) \neq \Vol(M) =\Vol_{M}(\chi_0).$ We will now show that points with this property are exceptional: the peripheral traces do determine the volume in the complement of a proper subvariety.

\begin{pro}\label{exactn-volume-form}
Denote $\Ei_0$ the component of $\Ei(M)$ corresponding to $\X_0.$
Then there is a proper subvariety $U$  of $\Ei_0$ such that the 1--form $\eta$ is exact on $\Ei_{0}\setminus U.$ 

Moreover, there is a proper subvariety $V' \subset \X_0$ containing the subvariety $V$ defined in (\ref{def:V}) with the property that the restriction $\Vol_M\co \X_0 \setminus V' \to \RR$ factors through a real valued function on $r(\X_0\setminus V').$ In particular, if $\chi_1, \chi_2 \in \X_0\setminus V'$ and $r(\chi_1) = r (\chi_2)$, then $\Vol_{M} (\chi_1) = \Vol_{M}(\chi_2)$.
\end{pro}

\begin{proof}
The quotient map $p\co \Ei_0\to \X(\partial M)$ is division by $\Gamma= \Z_2^h,$ where the $h$ $\Z_2$--factors are generated by $(m_i, l_i) \to (m_1^{-1}, l_i^{-1})$ for $1\le i \le h.$ The union of the fixed point sets is contained in the proper subvariety
$$U = \Ei_0\cap \bigcup_{i=1,\ldots,h} \{m_i^2=l_i^2=1\}.$$ 
Notice that $U$ corresponds to the subvariety $V$ of $\X_0$ defined in the previous section. 

Denote $p(\Ei_0) = Y_0 \subset \X(\partial M).$
Then $p^* \co H^1(Y_0\setminus p(U), \RR) \to H^1(\Ei_0\setminus U, \RR)$ is injective with image $H^1(\Ei_0\setminus U, \RR)^\Gamma$ (see, for instance, Hatcher~\cite{ha}, Proposition 3G.1).
The form $\eta$ is invariant under the action of $\Gamma,$ and hence $[\eta] \in H^1(\Ei_0\setminus U, \RR)^\Gamma.$ Whence there is a unique class $c\in H^1(Y_0\setminus p(U), \RR)$ that maps to $[\eta].$ 

The map $r\co \X_0\to Y_0$ has finite degree \cite{tillus_ei}, whence it is a branched cover and the branch set is contained in a proper subvariety of $\X_0.$ Denote $V'$ the union of this subvariety with $V,$ and note that $V$ is invariant under the covering transformations, so that $r\co \X_0\setminus V' \to Y_0$ is a finite cover onto its image. Let $W' = r(V').$ Since subvarieties have real co-dimension at least two, the restriction $r\co \X_0 \setminus V'\to Y_0 \setminus W'$ is a finite cover of connected topological spaces, and so 
$$r^*\co H^1(Y_0 \setminus W' , \RR) \to H^1(\X_0 \setminus V', \RR)$$ 
is an injection. The definitions of $\eta,$ $r$ and $p$ as well as Hodgson's formula for the volume differential imply that $\de\Vol_M\in r^*(c).$ Whence $r^*(c) = [\de \Vol_M] = 0.$ Since $r^*$ is an injection, we have $c=0,$ and so $[\eta]=p^*(c)=0.$ This completes the proof of the exactness statement.

Since the subvariety $U$ has real co-dimension at least 2, $[\eta]=0$ implies that there is a function $\Vol_\Ei\co \Ei_0\setminus U \to \RR$ with  $\de \Vol_\Ei = \eta,$ and which we normalise so that for some $x_1\in \Ei_0$ with $r(x_1) \in Y_0 \setminus W',$ we have $\Vol_\Ei(x_1) = \Vol_M(\chi),$ where $\chi$ satisfies $p(x_1) = r(\chi).$ Since $\eta$ is invariant under $\Gamma,$  there is a function $\Vol_\partial \co Y_0 \setminus W' \to \RR$ such that $\Vol_\Ei \co \Ei_0 \setminus U\to \RR$ factors through $\Vol_\partial.$ 

The last claim follows if we show that $\Vol_M \co \X_0 \setminus V'\to \RR$ also factors through $\Vol_\partial.$ This is done using the following construction from \cite{tillus_ei}. If $\X(M)$ is a variety in $\C^m,$ we now define the variety $\X_E(M)$ in $\C^m \times (\C\setminus \{0\})^{2h}$ by adding $[m_1^{\pm 1}, l_1^{\pm 1}, \ldots, m_h^{\pm 1}, l_h^{\pm 1}]$ to the coordinate ring of $\X(M)$ and adding the following generators to the ideal defining $\X(M):$
\begin{align*}
I_{\m_i} &= m_i + m^{-1}_i,\\
I_{\longitude_i} &= l_i + l^{-1}_i,\\
I_{\m_i\longitude_i} &= m_il_i + m^{-1}_il^{-1}_i,
\end{align*}
for $i=1,\ldots, h,$ noting that the left-hand sides are polynomials in the coordinates of $\X(M).$
The projections $p_2\co \X_E(M) \to \X(M)$ and $r_E \co \X_E(M) \to \Ei(M)$ are dominating maps. Every point in $\X_E(M)$ is of the form $(\chi, x),$ where $\chi\in \X(M)$ and $x\in \Ei(M).$ We define $\Vol\co \X_E(M) \to \RR$ by $\Vol(\chi, x) = \Vol_M(\chi).$ In these coordinates, Hodgson's work shows $\de\Vol = r_E^*(\eta)$ in the complement of the set of characters sending a peripheral subgroup to $\{\pm E\}.$ 
Denote $\X_{E,0}$ the component of $\X_E(M)$ corresponding to $\X_0.$
Since $\eta$ is exact on $\Ei_0\setminus U,$ it follows that $\Vol$ factors through $\Vol_\Ei,$ and hence through $\Vol_\partial,$ on (the complement of a proper subvariety in) $\X_{E,0}.$
Moreover, the action of $\Gamma$ on $\X_E(M)$ gives the quotient map $p_2\co \X_E(M) \to \X(M)$ and $\Vol$ is invariant under this action. Whence $\Vol_M$ also factors through $\Vol_\partial.$ This completes the proof.
\end{proof}


\section{A Zariski dense set}
\label{sec:A Zariski dense set}

Since $r : \PX_0 \to \PX(\partial M)$ is finite--to--one onto its image by Proposition 13 of
\cite{tillus_ei}, it suffices to show that $r$ is one--to--one over a Zariski
dense set $Z \subset r(\PX_0)$. We now determine a suitable set.  

Let $M$ be a complete hyperbolic 3--manifold with $h$ cusps and a chosen
orientation.  Choose a basis $\{ \m_i, \longitude_i\}$ for $H_1(T_i)$, where $T_i$ is
a torus cross section of the $i$--th cusp. Denote by $M_\kappa$ the oriented
3--manifold obtained by Dehn surgery on $M$ with coefficient $\kappa =
(p_1,q_1;...;p_h,q_h)$, where $(p_i,q_i)$ is either a co--prime pair of
integers or $\infty$.  Thurston showed that $M_\kappa$ has a complete
hyperbolic structure for all $\kappa$ in a neighbourhood $N$ of $\infty =
(\infty;...;\infty)$ in $S^2 \times...\times S^2$, and that $\lim_{\kappa \to
  \infty} \Vol (M_\kappa) = \Vol (M)$. Moreover, there is a unique discrete and
faithful character $\chi_0 \in \PX_0(M)$ which corresponds to the complete hyperbolic 
structure on $M$ with the chosen orientation, and a neighbourhood $U(\chi_0)$ of $\chi_0$ such
that if $\kappa \in N$ and $\chi_\kappa$ is the character of the holonomy of $M_\kappa$, 
then $\chi_\kappa \in U(\chi_0)$.
  
Let $0<<p_1<p_2<...<p_k<...$ be an infinite sequence of primes, $S=\{p_i\}$,
with the property that $S' = \{(1,q_1;...;1,q_h) \mid q_i \in S\} \subset N$.
In particular, $M_\kappa$ is a closed hyperbolic 3--manifold for each $\kappa
\in S'$. Let $W = \{ \chi_\kappa \mid \kappa \in S'\}$.

We claim that $W$ is a Zariski dense subset of $\PX_0$. Choose a lift $\tilde{\chi}_0$ of $\chi_0$ in the $\SL$--character variety and corresponding lifts $\X_0$ of $\PX_0,$ $\widetilde{U}$ of $U(\chi_0)$ and $\widetilde{W}\subset \widetilde{U}$ of $W$. Since the quotient map $\X\to\PX$ is finite--to--one, it suffices to show that $\widetilde{W}$ is a Zariski dense
subset of $\X_0$.

Given $\gamma\in\pi_1(M)$, there is a rational function $I_{\gamma}:\X_0\to\C$
defined by $\chi\to\chi(\gamma)$. Thurston shows in \cite{t}, \S5.8,
that the map $I :\X_0 \to \C^h$ defined by $I(\chi_{\rho}) =
(I_{\gamma_1},...,I_{\gamma_h})$ maps a small open neighbourhood of $\chi_0$
to an open neighbourhood of $(\epsilon_1 2,...,\epsilon_h 2)$, where
$I_{\gamma_{i}}(\tilde{\chi}_{0}) = \epsilon_i 2$ and $\gamma_i$ is a fixed
primitive element of $H_1(T_i)$.  This is only possible if the functions
$I_{\gamma_{i}}$ are algebraically independent over $\C$ as elements of
$\C[\X_{0}]$.

Assume that $\widetilde{W}$ is not Zariski dense in $\X_0$. We proceed by
complete induction on the number of cusps. If $h=1$, then $\widetilde{W}$ is a finite
collection of points. This is not possible since $\widetilde{W}$ contains the
holonomy characters of infinitely many pairwise non--isometric closed hyperbolic
manifolds because their volumes tend to the volume of $M$.

So assume that $\widetilde{W}$ is a Zariski dense subset of $\X_0(N)$ whenever
$N$ has fewer than $h$ cusps, and that the hypothesis fails for $M$, which has $h$
cusps. Thus, $\widetilde{W}$ is contained in a finite union of proper irreducible subvarieties
$\cup U_i \subset \X_0,$ where $\dim U_i \le h-1$. Performing surgery on
one cusp, say the first, gives an infinite family of complete hyperbolic
$(h-1)$--cusped manifolds $M_j= M(1, p_j;\infty,...;\infty)$. For each $M_j$, surgeries
on the remaining cusps with resulting coefficients contained in $S'$ give a Zariski
dense set in $X_0(M_j)$ by the hypothesis; whence each $X_0(M_j)$ must equal
some $U_i$, and in particular, infinitely many of these Dehn surgery components
are identical. By passing to a subsequence and renumbering, we may assume
that $U_1 = X_0(M_j)$ for each $M_j$. The discrete and faithful character of $M_j$ is contained in a finite set determined by the intersection of $U_1$
with the hypersurfaces $I_{\gamma_i}^2 = 4$, $i=2,...,h$. Thus, infinitely many
of the holonomy representations are conjugate, which is again not possible
since the surgery coefficients tend to $(\infty; ...; \infty)$, and hence
$\lim \Vol(M_j) = \Vol(M)$. This proves that $\widetilde{W}$ is a Zariski dense subset of $\X_0(M)$.

The quotient map $\qe: \X_0\to\PX_0$ and the restriction $r : \PX_0 \to
\PX(\partial M)$ are finite--to--one, and therefore the set $r(W)$ is a
Zariski dense subset of $r(\PX_0)$.
\begin{itemize}
\item[$\blacktriangleright$] Whence the set $Z=r(W)\setminus p(U),$ where $U$ is the subvariety from Proposition~\ref{exactn-volume-form}, is also Zariski dense in $r(\PX_0).$
\item[$\blacktriangleright$] Moreover, $Z$ has the property that if $z
\in Z$, then there is a character $\chi\in r^{-1}(z)$ which is the character
of a holonomy of a closed hyperbolic manifold obtained by Dehn filling on $M$.
\item[$\blacktriangleright$] In particular, this character is unique (up to complex conjugation) by Mostow rigidity.
\end{itemize}


\section{Proofs of the main results}
\label{sec:Proofs}

\begin{proof}[Proof of Theorem \ref{thm:birational}]
Assume that $M$ has $h$ cusps. It suffices to show that each point in the set $Z \subset r(\PX_0)$ of the previous section has exactly one preimage. By construction, for each $z \in Z$, there is a closed hyperbolic 3--manifold $N=M(\gamma_1,...,\gamma_h)$ obtained by Dehn filling on $M$ and a preimage $\chi \in r^{-1}(z)$ such that $\chi$ is the character of a holonomy for the complete hyperbolic structure on $N$.  In particular, we have $\Vol_{M}(\chi)=\Vol_{N}(\chi) = \pm \Vol(N)$ by Lemma~\ref{volume-factors}.  
  
Now assume that $\chi' \in r^{-1}(z)$ is another preimage. Proposition~\ref{exactn-volume-form}, yields $\Vol_{M}(\chi) = \Vol_{M}(\chi')$. We claim that $\chi'$ also factors through $\pi_1(N)$. The characters of peripheral elements are completely determined by $z$.  Thus, each peripheral subgroup has a non--trivial rotation in its image and the curves $\gamma_i$ are represented by parabolics.  Since a peripheral subgroup is abelian, the images of the $\gamma_i$ must be trivial, and hence $\chi'$ factors through $\pi_1(N)$.  Lemma \ref{volume-factors} may now be applied:
\begin{equation*}
     \Vol_{N}(\chi') = \Vol_{M}(\chi') = \Vol_{M}(\chi)=\Vol_{N}(\chi) = \pm \Vol(N).
\end{equation*}
Thus by Theorem~\ref{volume-complete}, $\chi'$ is a discrete and faithful character corresponding to the complete structure on $N$, and hence by Mostow Rigidity either $\chi'=\chi$ or $\chi'=\overline{\chi}$. Since complex conjugation reverses orientation and hence changes the sign of the volume function, we have $\Vol_{N}(\chi') = -\Vol_{N}(\chi)$ in the second case, which implies an offending statement: $\Vol(N)=0$.
\end{proof}

\begin{proof}[Proof of Corollary \ref{thm:birational}]
The proof of Dunfield, \cite{du2} Corollary 3.2, applies almost verbatim. A representation $\rho\co \pi_1(M)\to \PSL$ has $|H^1(M, \Z_2)|$ pairwise distinct lifts to $\SL.$ Now $H_1(\partial M, \Z_2) \cong \Z_2^{2h}$ and $\text{im}(H_1(\partial M, \Z_2) \to H_1(M, \Z_2)) \cong \Z_2^h.$ So by duality, if $|H^1(M, \Z_2)| = 2^{h+k},$ then the image of these $2^{h+k}$ representations consists (generically) of exactly $2^h$ points in $\X(\partial M).$ The genericity hypothesis applies to $\rho_0.$ Since distinct lifts of $\rho_0$ may be on distinct components of the $\SL$--character variety, the degree of $r\co\X_0 \to \X(\partial M)$ is \emph{at most} $2^k.$
\end{proof}



\address{Mathematics, Statistics, and Computer Science, University of Illinois at Chicago, Chicago, IL 60607-7045, USA}
\email{bklaff1@uic.edu}

\address{School of Mathematics and Statistics, The University of Sydney, NSW 2006, Australia} 
\email{tillmann@maths.usyd.edu.au}
\Addresses


\end{document}